\documentclass{amsart}

\usepackage[hiresbb]{graphicx}
\usepackage{amsmath,amssymb}
\usepackage{amsthm}
\usepackage{amsfonts}
\usepackage{amscd}
\newtheorem{df}{Definition}[section]
\newtheorem{thm}[df]{Theorem}
\newtheorem{prop}[df]{Proposition}
\newtheorem{lemm}[df]{Lemma}
\newtheorem{cor}[df]{Corollary}

\newtheorem{fact}[df]{Fact}
\newtheorem{expectation}[df]{Expectation}

\newcommand{\Q}{\mathbb{Q}}

\newcommand{\Z}{\mathbb{Z}}

\newcommand{\shuugou}[1]{\{ #1 \}}
\newcommand{\zettaiti}[1]{\lvert #1 \rvert}

\newcommand{\gyaku}[1]{ #1^{-1}}

\newcommand{\skein}[1]{\mathcal{S}( #1 )}

\newcommand{\defeq}{\stackrel{\mathrm{def.}}{=}}
\newcommand{\Aut}{\mathrm{Aut}}
\newcommand{\bch}{\mathrm{bch}}

\newcommand{\filtn}[1]{\{ #1 \}_{n \geq 0}}

\newcommand{\comp}[1]{\underleftarrow{\lim}_{#1 \rightarrow \infty}}

\newcommand{\gauss}[1]{\lfloor #1 \rfloor}
\newcommand{\arccosh}{\mathrm{arccosh}}
\newcommand{\Poincare}{Poincar\'{e} }

\begin{document}

\title[ An invariant for $\Z HS^3$ via skein algebras]
{Construction of an invariant for integral homology 3-spheres via completed Kauffman
bracket skein algebras }
\author{Shunsuke Tsuji}
\date{}
\maketitle

\begin{abstract}
We construct an invariant $z (M) =1+a_1(A^4-1)+
a_2(A^4-1)^2+a_3(A^4-1)^3 + \cdots \in \Q [[A^4-1]]=\Q [[A+1]]$
for an integral homology $3$-sphere $M$ using
a completed skein algebra and a Heegaard splitting.
The invariant $z(M)\mathrm{mod} ((A+1)^{n+1}) $
is a finite type invariant of order $n$.
In particular, $-a_1/6$ equals the 
Casson invariant.
If $M$ is the Poincar\'{e} homology 3-sphere,
$(z(M))_{|A^4 =q} \mod (q+1)^{14}  $ is the Ohtsuki series \cite{Ohtsuki1995} for $M$.
\end{abstract}

\section{Intoduction }
Heegaard splitting theory clarifies a
relationship between  mapping class groups
on surfaces and  closed oriented 3-manifolds.
In particular, there exists some equivalence relation $\sim$
of Torelli groups of a surface $\Sigma_{g,1}$
with genus $g$ and non-empty connected boundary,
and the well-defined bijective map
\begin{equation*}
\lim_{g \to \infty} \mathcal{I}(\Sigma_{g,1})/ {\sim} \to \mathcal{H} (3)
\end{equation*}
plays an important role, where we denote by  $\mathcal{I}(\Sigma_{g,1})$
the Torelli group of $\Sigma_{g,1}$
 and by $\mathcal{H}(3)$ the set of integral homology 
$3$-spheres, i.e. closed oriented $3$-manifolds
whose homology groups are isomorphic to
the homology group of $S^3$.
For details, see Fact \ref{fact_map_3_manifold} in this paper.
This bijective map makes it possible to study
integral homology $3$-spheres using the structure of Torelli groups.
See, for example, Morita \cite{Morita1989} and
Pitsch \cite{Pitsch2008} \cite{Pitsch2009}.

On the other hand, in
 our previous papers \cite{TsujiCSAI} \cite{Tsujipurebraid} \cite{TsujiTorelli},
we study some new relationship between the Kauffman bracket skein algebra and
the mapping class group of a surface.
It gives us a new way
of studying the mapping class group.
For example \cite{TsujiTorelli}, 
we reconstruct the first Johnson homomorphism
in terms of the skein algebra.
Since the Kauffman bracket skein algebra  comes from 
 link theory, we expect that this relationship
 brings us a new information
of $3$-manifolds. 

The aim of this paper is to construct  an invariant $z(M)$
 for an integral homology
$3$-sphere $M$ using completed skein algebras
and the above bijective map.
In other words, the aim of this paper
is to prove the following main theorem.
\begin{thm}[Theorem \ref{thm_main}]
The map $Z: \mathcal{I}(\Sigma_{g,1}) \to \Q[[A+1]]$
defined by 
\begin{equation*}
Z(\xi) \defeq \sum_{i=0}^\infty \frac{1}{(-A+\gyaku{A})^i i!}e_*
((\zeta (\xi))^i)
\end{equation*}
induces an invariant
\begin{equation*}
z:\mathcal{H} (3) \to \Q[[A+1]], M(\xi) \to Z(\xi),
\end{equation*}
where $e_*$ is the $\Q [[A+1]]$-module homomorphism
induced  by standard embedding.
Here $\zeta :\mathcal{I} (\Sigma_{g,1}) \to
\widehat{\skein{\Sigma_{g,1}}}$ is an embedding
defined in Theorem \ref{thm_zeta}.
\end{thm}
We remark we do not rely on number theory
for constructing the invariant.

Let $V$ be a $\Q$-vector space.
In our paper, a map $z':\mathcal{H}(3) \to V$
is called a finite type invariant of 
order $n$ if and only if
the $\Q$-linear map
$z':\Q \mathcal{H}(3) \to V$
induced by $z':\mathcal{H}(3) \to V$
satisfies the condition that
\begin{equation*}
\sum_{\epsilon_i \in \shuugou{0,1}} (-1)^{\sum \epsilon_i}z'(M(\prod_{i=1}^{2n+2} {\xi_i}^{\epsilon_i}))=0.
\end{equation*}
for any $\xi_1,\xi_2, \cdots, \xi_{2n+2} \in 
\mathcal{I}(\Sigma_{g,1})$.
The above condition and the condition
that
\begin{equation*}
\sum_{\epsilon_i \in \shuugou{0,1}} (-1)^{\sum \epsilon_i}z'(M(\prod_{i=1}^{n+1} {\xi_i}^{\epsilon_i}))=0.
\end{equation*}
for any $\xi_1,\xi_2, \cdots, \xi_{n+1} \in 
\mathcal{K}(\Sigma_{g,1})$
are equivalent
to each other.
This follows from
 \cite{GL1997} Theorem 1 and \cite{GGP2001} subsection 1.8.
Furthermore, in our paper,
a finite type invariant 
$z': \mathcal{H}(3) \to V$
of order $n$
is called nontrivial
if and only if
the $\Q$-linear map
$z':\Q \mathcal{H}(3) \to V$
induced by $z':\mathcal{H}(3) \to V$
satisfies the condition that
there exists $\xi_1, \xi_2, \cdots, \xi_{2n} \in
\mathcal{I}(\Sigma_{g,1})$ such  that
\begin{equation*}
\sum_{\epsilon_i \in \shuugou{0,1}} (-1)^{\sum \epsilon_i}z'(M(\prod_{i=1}^{2n} {\xi_i}^{\epsilon_i})) \neq 0.
\end{equation*}
By \cite{GL1997} Theorem 1 and \cite{GGP2001} subsection 1.8,
the above condition and the condition
that there exists $\xi_1, \xi_2, \cdots, \xi_{n} \in
\mathcal{K}(\Sigma_{g,1})$ such  that
\begin{equation*}
\sum_{\epsilon_i \in \shuugou{0,1}} (-1)^{\sum \epsilon_i}z'(M(\prod_{i=1}^{n} {\xi_i}^{\epsilon_i})) \neq 0.
\end{equation*}
are equivalent to each other.
The invariant $z: \mathcal{H}(3) \to \Q[[A+1]]$
defined in this paper induces 
a finite type invariant 
$z(M) \in \Q[[A+1]]/((A+1)^n)$ of order $n+1$
for $M \in \mathcal{H} (3)$ (Corollary \ref{cor_finite_type}).
In Proposition \ref{prop_z_finite_nontrivial},
we prove
the finite type invariant
$z(M) \in \Q[[A+1]]/((A+1)^n)$ of order $n+1$
is nontrivial,
where we use 
a connected sum of the \Poincare spheres.

Furthermore,
we give some computations of  this invariant $z$ for some 
integral homology 3-spheres.
As a corollary of this computation, the coefficient of $(A^4-1)$ in $z$ is $(-6)$ times
the Casson invariant.
On the other hand, Ohtsuki \cite{Ohtsuki1995}
defined the Ohtsuki series $\tau:\mathcal{H}(3) \to
\Z [[q]]$.
If $M$ is the \Poincare homology $3$-sphere,
$z(M) \mod ((A+1)^{14})$ is equal to 
$ \tau (M)_{|q=A^4} \mod ((A+1)^{14})$.
This lead us the following.

\begin{expectation}
Using the change of variables $A^4=q$,
the invariant $z$ induces
the Ohtsuki series $ \tau$,
in other words, we have $z(M) = \tau (M)_{|q=A^4}$ for any $M \in
\mathcal{H}(3)$.
\end{expectation}

\tableofcontents

\section*{Acknowledgements}
The author would like to thank Kazuo Habiro, Gw\'{e}na\"{e}l Massuyeau, Jun Murakami
and Tomotada Ohtsuki for
helpful comments about finite type invariatns of integral homology 
3-spheres.
The author must make special mention of Genki Omori,
who explain the works
of Pitsch \cite{Pitsch2009}
in order to prove  Lemma \ref{lemm_torelli_check}.
This work was supported by JSPS KAKENHI Grant Number 15J05288 
and the Leading Graduate Course for Frontiers of Mathematical Sciences and Phsyics.

\section{Mapping class grups and closed 3-manifolds}

Let $\Sigma_g$ denote an closed oriented surface of genus 
$g$ standardly embedded in the oriented $3$-sphere $S^3$.
The embedded surface $\Sigma_g$ separates $S^3$ into two
handle bodies of genus $g$, $S^3 =H_g^+ \cup_\varphi H_g^-$
where $\varphi : \Sigma_g =\partial H_g^+ \to \partial H_g^-$
is a diffeomorphism. We fix an closed disk $D$ in $\Sigma_g$
and denote by $\Sigma_{g,1}$ the closure of $ \Sigma_g \backslash D$.
The embedding $\Sigma_g \hookrightarrow S^3$ 
determines two natural subgroups of 
\begin{equation*}
\mathcal{M}(\Sigma_{g,1})
\defeq \mathrm{Diff}^+ (\Sigma_{g,1}, \partial \Sigma_{g,1})/
\mathrm{Diff}_0(\Sigma_{g,1}, \partial \Sigma_{g,1}),
\end{equation*}
namely
\begin{equation*}
\mathcal{M}(H_{g,1}^\epsilon)
\defeq \mathrm{Diff}^+ (H_{g}^\epsilon, D)/
\mathrm{Diff}_0(H_{g}^\epsilon, D).
\end{equation*}
for $\epsilon \in \shuugou{+,-}$.
We denote $M(\xi) \defeq H_g^+ \cup_{\varphi \circ \xi} H_g^-$.
Let $\mathcal{I} (\Sigma_{g,1}) $ be the Torelli group of the surface $\Sigma_{g,1}$,
which is the set consisting of all elements of $\mathcal{M}(\Sigma_{g,1})$
acting trivially on $H_1(\Sigma_{g,1})$.
We remark that
there is a natural injective stabilization map $\mathcal{M}(\Sigma_{g,1}) \hookrightarrow
\mathcal{M}(\Sigma_{g+1,1})$, which is
compatible with the definitions of the above two subgroups.

\begin{df}
For $\xi_1$ and $\xi_2 \in \mathcal{I}(\Sigma_{g,1})$, we define
$\xi_1 \sim \xi_2$ if there exist
$\eta^+ \in \mathcal{M}(H_{g,1}^+)$ and $\eta^- \in \mathcal{M}(H_{g,1}^-)$
satisfying $\xi_1 =\eta^- \xi_2 \eta^+$.
\end{df}

\begin{fact}[For example, see \cite{Morita1989} \cite{Pitsch2008}\cite{Pitsch2009}]
\label{fact_map_3_manifold}
The map 
\begin{equation*}
\underrightarrow{\lim}_{g \rightarrow \infty} (\mathcal{I}(\Sigma_{g,1})/\sim )
\to \mathcal{H}(3),\xi \mapsto M(\xi)
\end{equation*}
is bijective, where $\mathcal{H}(3)$ is the set of integral homology
$3$-spheres, i.e.,closed oriented $3$-manifolods
 whose homology group is isomorphic to
the homology group of $S^3$.
\end{fact}

We denote $\mathcal{IM} (H_{g,1}^\epsilon)
\defeq \mathcal{I} (\Sigma_{g,1}) \cap \mathcal{M}(H_{g,1}^\epsilon)$
for $\epsilon \in \shuugou{+,-}$.

\begin{lemm}[ Pitsch \cite{Pitsch2009}, Theorem 9, P.295, Omori \cite{Omori2016}]
\label{lemm_torelli_handle_generator}
For $\epsilon \in \shuugou{+,-}$, the subgroup $\mathcal{IM} (H_{g,1}^\epsilon)$ is generated by
\begin{equation*}
\shuugou{t_{\xi(c_a) \xi( c'_a)}|\xi \in \mathcal{M} (H_{g,1}^\epsilon)}
\cup \shuugou{t_{\xi(c_b) \xi(c'_b)}|\xi \in \mathcal{M} (H_{g,1}^\epsilon)},
\end{equation*}
where the simple closed curves $c_a ,c'_a,c_b$ and $c'_b$
are as in Figure \ref{figure_handle_torelli_generator}.

\end{lemm}

\begin{figure}
\begin{picture}(300,140)
\put(0,0){\includegraphics[width=300pt]{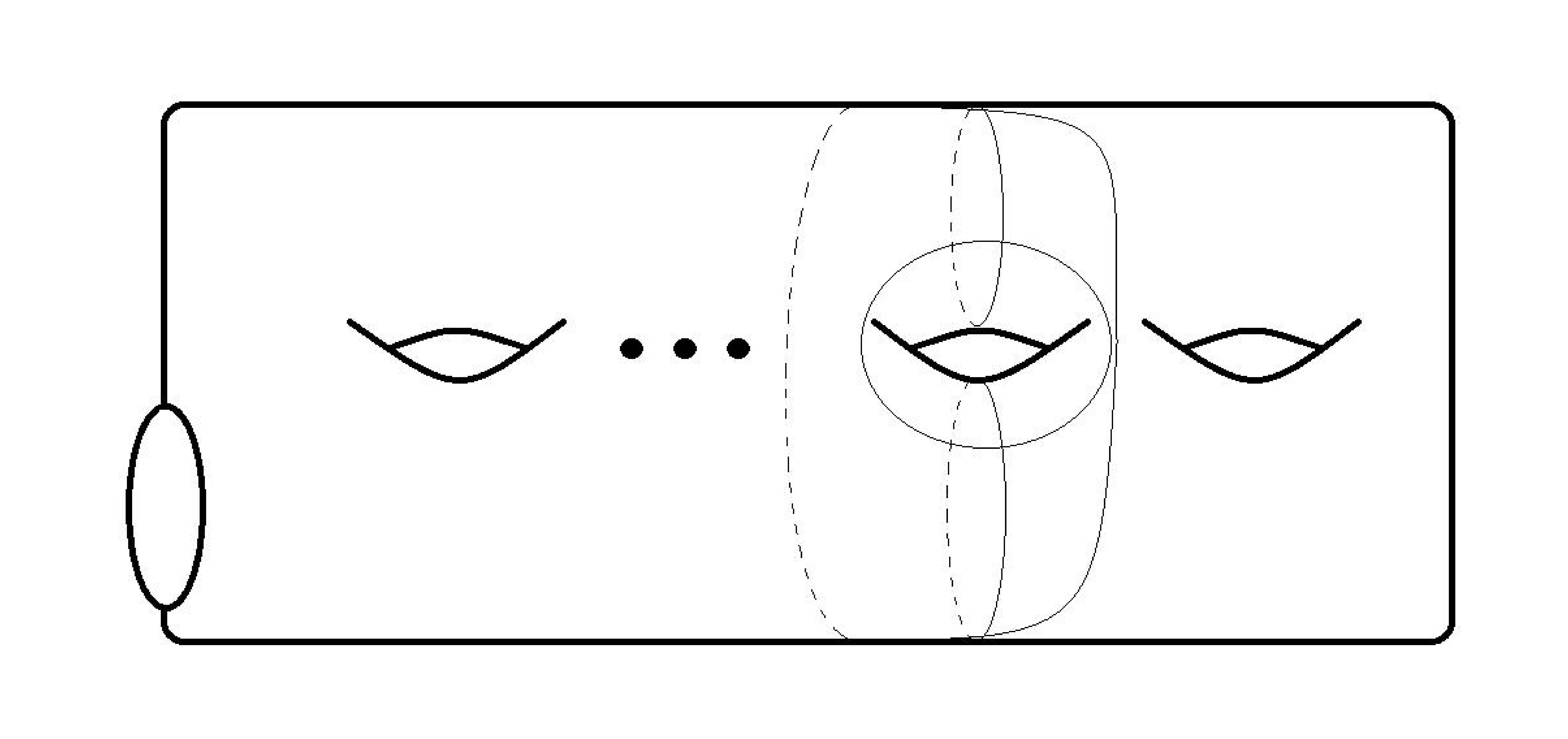}}
\put(37,90){$H_g^+$}
\put(183,112){$c_a$}
\put(183,30){$c'_a$}
\put(200,63){$c_b$}
\put(212,30){$c'_b$}
\put(14,90){$H_g^-$}
\end{picture}
\caption{$c_a ,c'_a,c_b$ and $c'_b$}
\label{figure_handle_torelli_generator}
\end{figure}

\begin{proof}
We prove the lemma in the case $\epsilon$ is $+$.
Let $I \Aut \pi_1(H_g^+,*)$ be the kernel of
$\Aut \pi_1(H_g^+,*) \twoheadrightarrow \Aut (H_1 (H_g))$.
By \cite{MKS}, Theorem N4, p.168,
$I \Aut \pi_1 (H_g^+,*)$ is generated by
\begin{equation*}
\shuugou{x_* \in I \Aut \pi_1(H_g^+,*)
|x \in \shuugou{t_{\xi(c_b) \xi(c'_b)}|\xi \in \mathcal{M} (H_{g,1}^+)}},
\end{equation*}
where we denote by $x_*$
the element of $I \Aut \pi_1(H_g^+,*)$ induced by $x
\in \shuugou{t_{\xi(c_b) \xi(c'_b)}|\xi \in \mathcal{M} (H_{g,1}^+)}$.
We denote by $\mathcal{LIM} (H_{g,1}^+)$ the Luft-Torelli group
which is the kernel of
$\mathcal{IM} (H_{g,1}^+) \to I \Aut \pi_1(H_g^+,*)$.
Pitsch \cite{Pitsch2009}, Theorem 9, P.295 proves that
$\mathcal{LIM} (H_{g,1}^+)$ is generated by
\begin{equation*}
\shuugou{t_{\xi(c_a) \xi( c'_a)}|\xi \in \mathcal{M} (H_{g,1}^+)}.
\end{equation*}
This proves the case that $\epsilon$ is $ +$.
If $\epsilon$ is $-$, replacing $a$ by $b$, 
the same proof works. 
This finishes the proof.
\end{proof}

\begin{lemm}[\cite{Pitsch2009}, Lemma 4, p.285]
Let $G$ be a subgroup of $\mathcal{M}(H_{g,1}^+)
\cap \mathcal{M}(H_{g,1}^-)$ such that
the natural map $G \to \Aut (H_1(H_{g,1}^+))$ is onto.
For two elements $\xi_1$ and $\xi_2 \in \mathcal{I} (\Sigma_{g,1})$,
$\xi_1 \sim \xi_2$ if and only if 
there exist $\eta_G \in G$, $\eta^+ \in \mathcal{IM} (H_{g,1}^+)$
and $\eta^- \in \mathcal{IM} (H_{g,1}^-)$ satisfying
$\eta^- \eta_G \xi_1 {\eta_G}^{-1} \eta^+ =\xi_2$.
\end{lemm}

\begin{proof}
Pitsch proved the above claim in the case
$G=\mathcal{M}(H^{+}_{g,1}) \cap
\mathcal{M}(H^{-}_{g,1}) $.
The proof is based on the fact
that the natural map $\mathcal{M}(H^{+}_{g,1}) \cap
\mathcal{M}(H^{-}_{g,1}) \to \Aut (H_1 (H^+_{g,1}))$ is  onto.
Therefore, the proof of \cite{Pitsch2009} Lemma 4 works for this lemma.
\end{proof}

We construct a subgroup of $\mathcal{M}(H^{+}_{g,1}) \cap
\mathcal{M}(H^{-}_{g,1}) $ satisfying the above condition.
Let $G \subset \mathcal{M}(H_{g,1}^+)
\cap \mathcal{M}(H_{g,1}^-)$  be the subgroup generated by
\begin{equation*}
\shuugou{h_i|i \in \shuugou{1,2, \cdots, g}} \cup
\shuugou{s_{ij}|i \neq j}
\end{equation*}
where we denote by $h_i$ and $s_{ij}$
the half twist along $c_{h,i}$ as in Figure \ref{figure_handle_G_h_i}
and the element $t_{c_{i,j}}{t_{c_{a,i}}}^{-1}{t_{c_{b,j}}}^{-1}$
as in Figure \ref{figure_handle_G_S_j_i}
and Figure \ref{figure_handle_G_S_i_j}.
Since this subgroup  $G$ satisfies the condition in the above lemma,
we have the following.

\begin{figure}
\begin{picture}(300,140)
\put(0,0){\includegraphics[width=300pt]{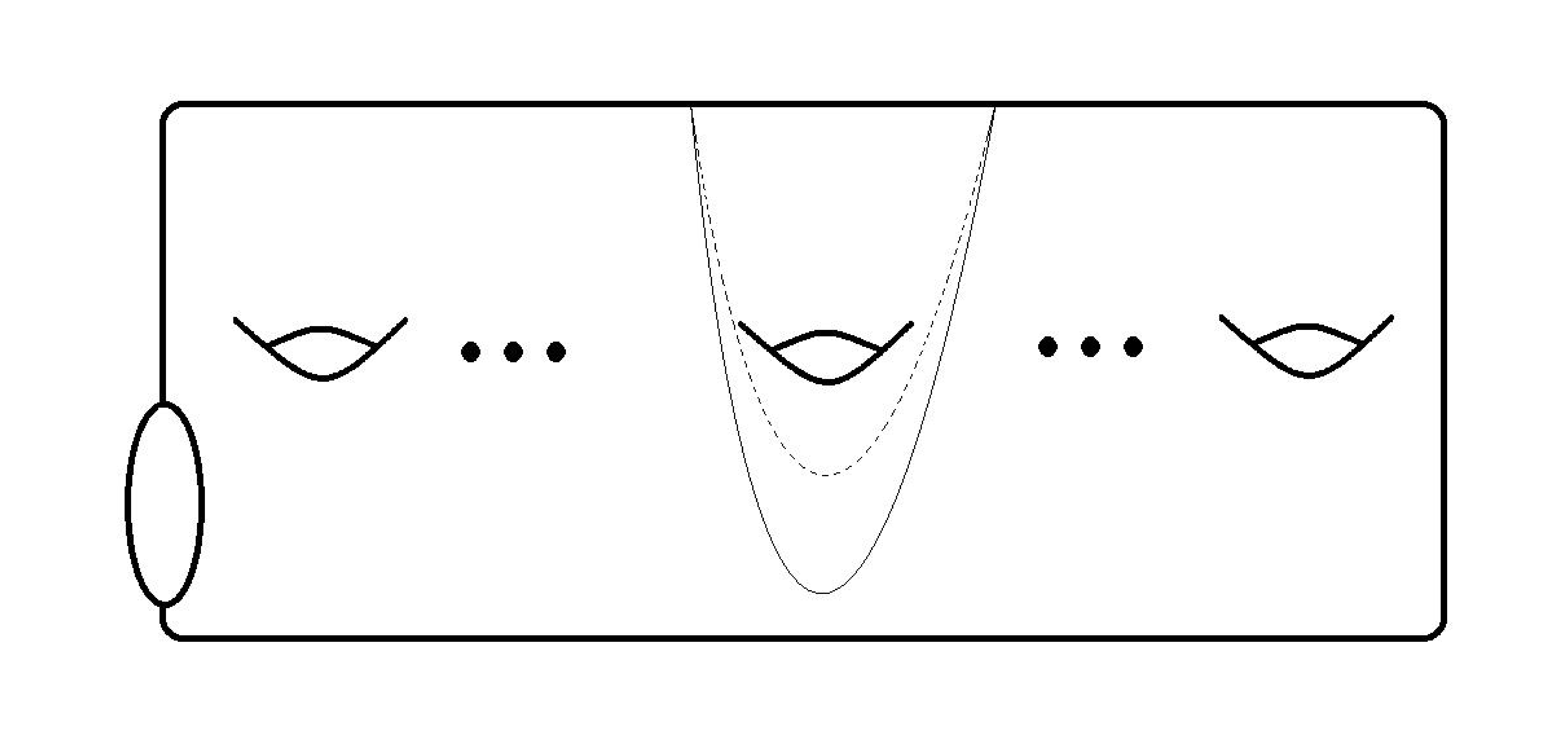}}
\put(37,90){$H_g^+$}
\put(57,90){$g$}
\put(250,90){$1$}
\put(160,90){$i$}
\put(180,63){$c_{h,i}$}
\put(14,90){$H_g^-$}
\end{picture}
\caption{$c_{h,i}$}
\label{figure_handle_G_h_i}
\end{figure}

\begin{figure}
\begin{picture}(300,140)
\put(0,0){\includegraphics[width=300pt]{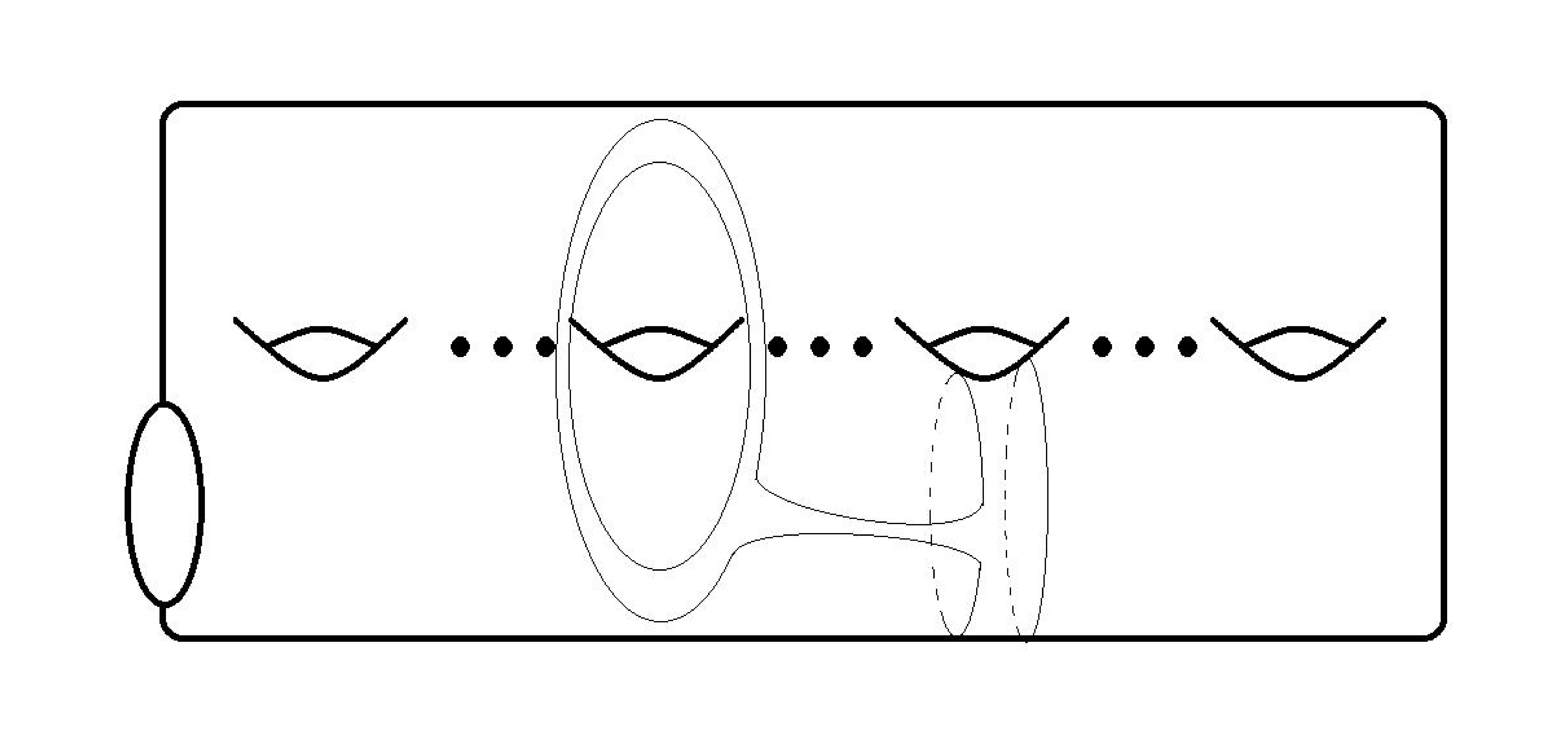}}
\put(37,90){$H_g^+$}
\put(57,90){$g$}
\put(250,90){$1$}
\put(204,43){$c_{a,i}$}
\put(160,27){$c_{i,j}$}
\put(120,45){$c_{b,j}$}
\put(187,90){$i$}
\put(125,90){$j$}
\put(14,90){$H_g^-$}
\end{picture}
\caption{$c_{a,i}$, $c_{b,j}$ and $c_{i,j}$ for $j>i$}
\label{figure_handle_G_S_j_i}
\end{figure}

\begin{figure}
\begin{picture}(300,140)
\put(0,0){\includegraphics[width=300pt]{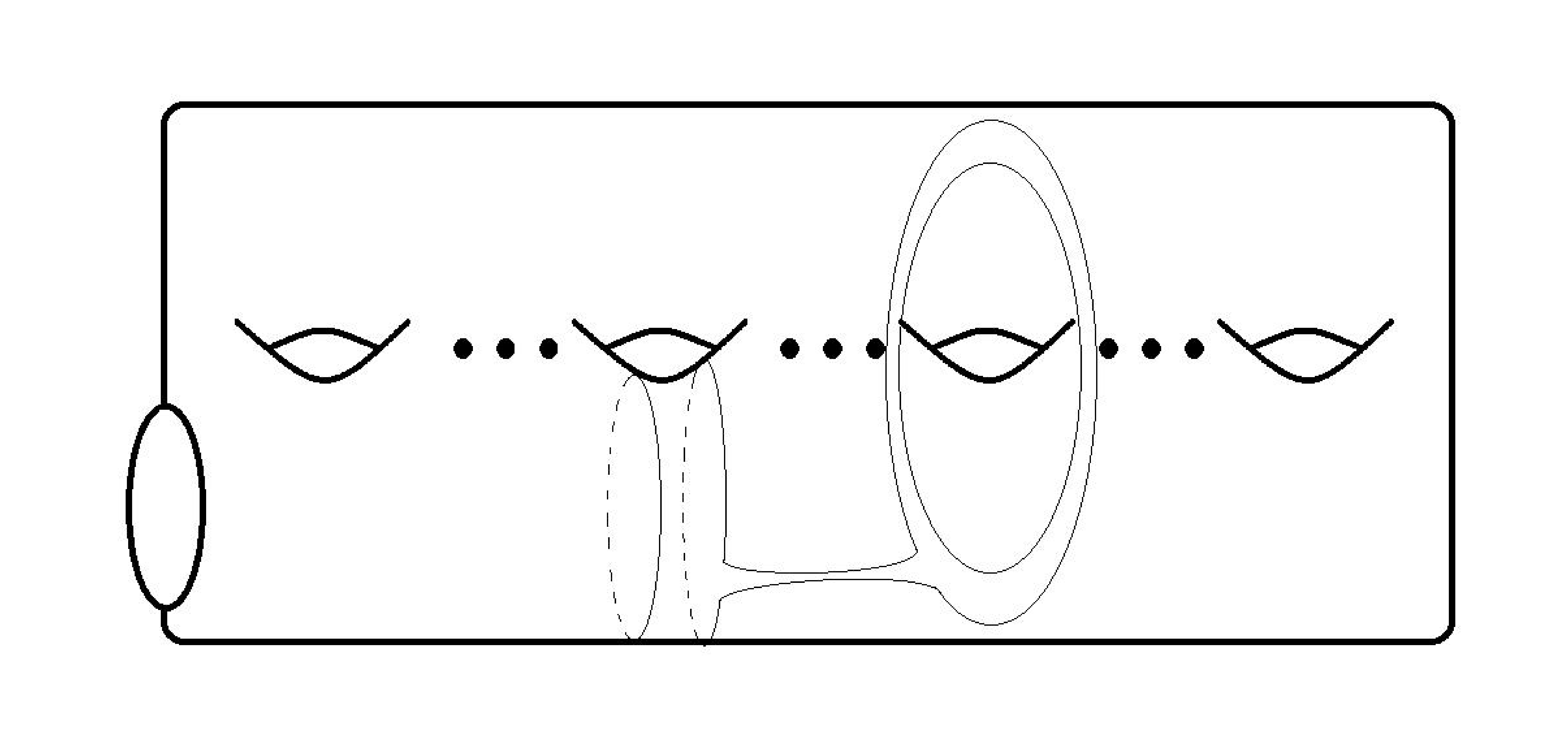}}
\put(37,90){$H_g^+$}
\put(57,90){$g$}
\put(250,90){$1$}
\put(213,63){$c_{i,j}$}
\put(180,53){$c_{b,j}$}
\put(100,43){$c_{a,i}$}
\put(187,90){$j$}
\put(125,90){$i$}
\put(14,90){$H_g^-$}
\end{picture}
\caption{$c_{a,i}$, $c_{b,j}$ and $c_{i,j}$ for $i>j$}
\label{figure_handle_G_S_i_j}
\end{figure}

\begin{lemm}
\label{lemm_torelli_check}
The equivalence relation $\sim$ in $\mathcal{I}(\Sigma_{g,1})$ is generated by
$\xi \sim \eta_{G} \xi {\eta_{G}}^{-1}$ for $\eta_G \in \shuugou{h_i, s_{i,j}}$,
$\xi \sim \xi \eta^+$ for $\eta^+ \in 
\shuugou{t_{\xi(c_a) \xi( c'_a)}|\xi \in \mathcal{M} (H_{g,1}^+)}
\cup \shuugou{t_{\xi(c_b) \xi(c'_b)}|\xi \in \mathcal{M} (H_{g,1}^+)}$
and 
$\xi \sim \eta^- \xi$ for $\eta^- \in 
\shuugou{t_{\xi(c_a) \xi( c'_a)}|\xi \in \mathcal{M} (H_{g,1}^-)}
\cup \shuugou{t_{\xi(c_b) \xi(c'_b)}|\xi \in \mathcal{M} (H_{g,1}^-)}$
\end{lemm}

\section{Proof of main theorem}
\subsection{Completed Kauffman bracket skein algebras and
Torelli groups} 
Let $\Sigma$ be a compact connected oriented surface.
We denote by $\mathcal{T}(\Sigma)$ the set of unoriented framed tangles in
$\Sigma \times I$.
Let $\skein{\Sigma}$ be the quotient of $\Q [A.\gyaku{A}] \mathcal{T}(\Sigma)$ 
by the skein relation and the trivial knot relation as in Figure 
\ref{figure_def_skein}.  We consider the product of $\skein{\Sigma}$ as in Figure
\ref{figure_def_product} and the Lie bracket $[x,y] \defeq
\frac{1}{-A+\gyaku{A}} (xy-yx)$ for $x,y \in \skein{\Sigma}$.
The completed Kauffman bracket skein algebra 
is defined by
\begin{equation*}
\widehat{\skein{\Sigma}} \defeq \comp{i}{\skein{\Sigma}/(\ker \varepsilon)^i}
\end{equation*}
where the augmentation map $ \varepsilon:\skein{\Sigma} \to \Q$ is defined by
$A+1 \mapsto 0$ and $[L]- (-2)^{\zettaiti{L}}\mapsto 0$ for $L \in \mathcal{T}(\Sigma)$.
In \cite{Tsujipurebraid}, we define the filtration
$\filtn{F^n \widehat{\skein{\Sigma}}}$ satisfying
\begin{align*}
&F^n \widehat{\skein{\Sigma}} F^m \widehat{\skein{\Sigma}} \subset
F^{n+m} \widehat{\skein{\Sigma}}, \\
&[F^n \widehat{\skein{\Sigma}}, F^m \widehat{\skein{\Sigma}}] \subset
F^{n+m-2} \widehat{\skein{\Sigma}}, \\
&F^{2n} \widehat{\skein{\Sigma}} = (\ker \varepsilon)^n.
\end{align*}
By the second equation, we can consider the Baker Campbell Hausdorff
series
\begin{equation*}
\bch(x,y) \defeq (-A+\gyaku{A}) \log ( \exp (\frac{x}{-A+\gyaku{A}})
\exp(\frac{y}{-A+\gyaku{A}}))
\end{equation*}
on $F^3 \widehat{\skein{\Sigma}}$.
As elements of the associated Lie algebra
$(\widehat{\skein{\Sigma}},[ \ \ , \ \ ])$,
$\bch$ has a usual expression.
For example, 
\begin{equation*}
\bch(x,y) = x+y+\frac{1}{2}[x,y]+\frac{1}{12}([x,[x,y]]+[y,[y,x]])+ \cdots.
\end{equation*}

Furthermore, we have the following.

\begin{prop}[\cite{Tsujipurebraid} Corollary 5.7.]
\label{prop_embed_filtration}
For any embedding $i : \Sigma \times I \to S^3$ inducing $i_*
: \widehat{\skein{\Sigma}} \to \Q [[A+1]]$,
we have $i_* (F^n \widehat{\skein{\Sigma}}) \subset  ((A+1)^{\gauss{\frac{n+1}{2}}})$,
where $\gauss{x}$ is the greatest integer not greater than
$x$ for $x \in \Q$. 
\end{prop}

\begin{figure}
\begin{picture}(300,200)
\put(0,0){\includegraphics[width=400pt]{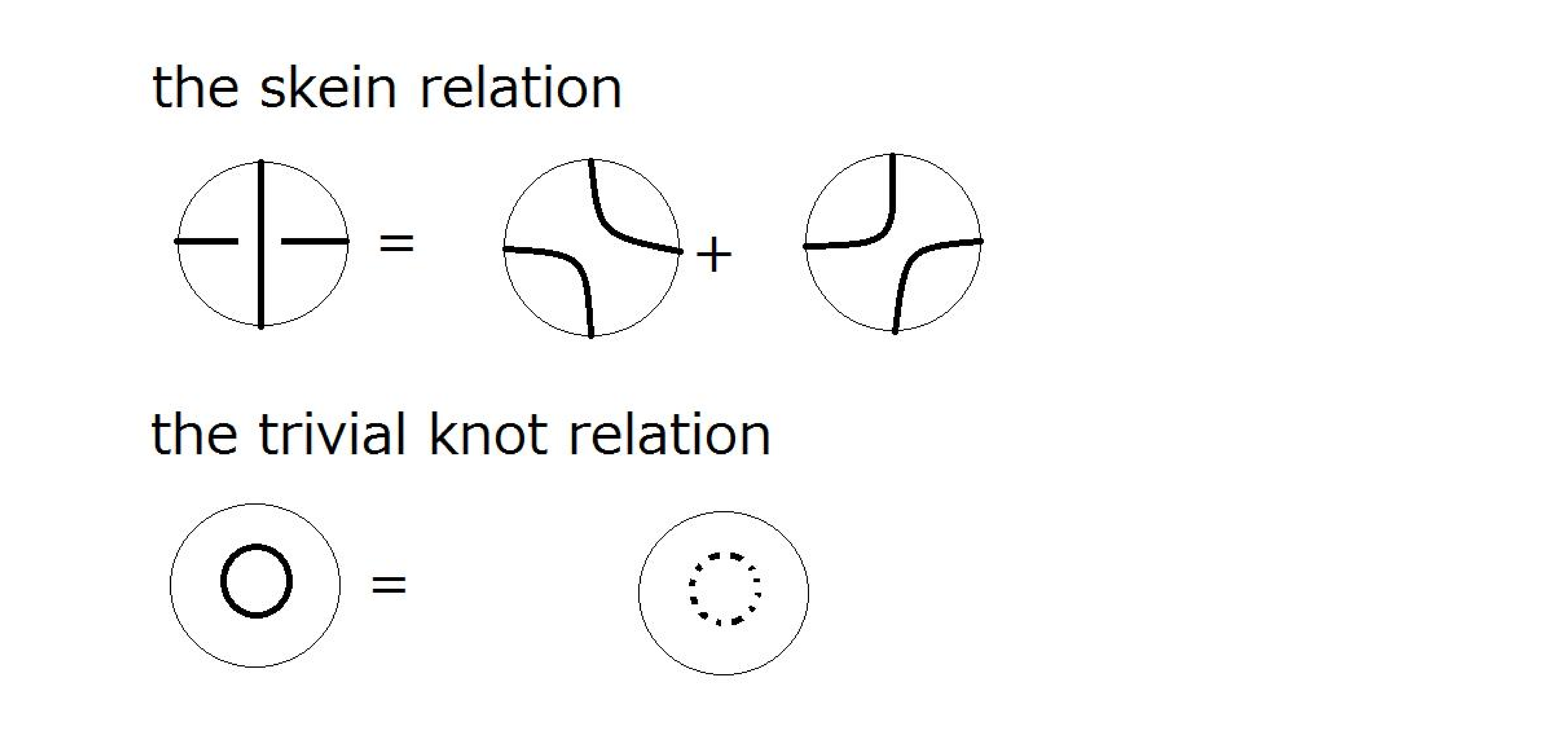}}
\put(110,120){$A$}
\put(187,120){$\gyaku{A}$}
\put(105,35){$(-A^2-A^{-2})$}
\end{picture}
\caption{Definition of Kauffman bracket skein module}
\label{figure_def_skein}
\end{figure}

\begin{figure}
\begin{picture}(140,70)
\put(0,0){\includegraphics[width=140pt]{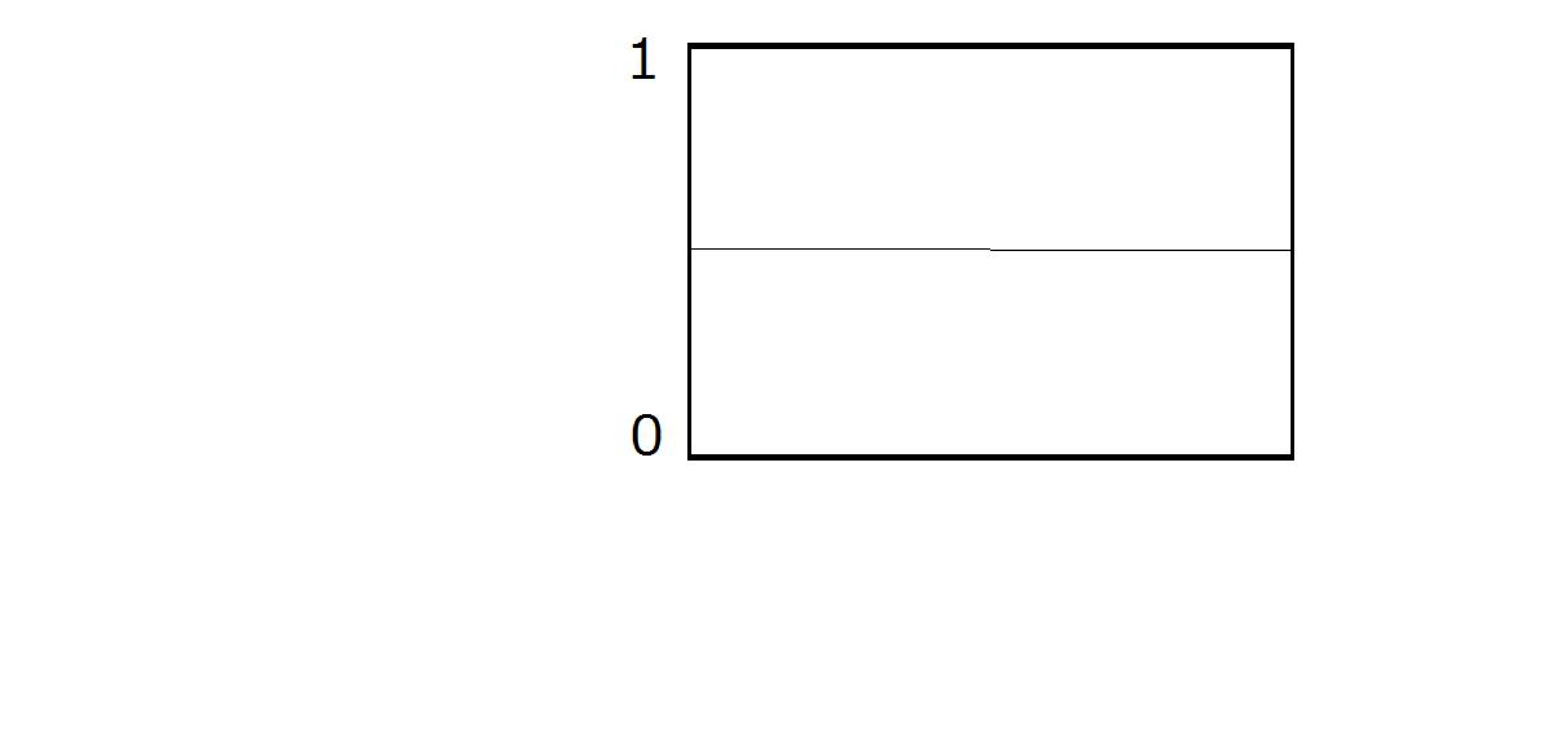}}
\put(0,40){$xy \defeq$}
\put(80,50){$x$}
\put(80,30){$y$}
\put(80,13){$\Sigma$}
\put(40,0){for $x, y \in \skein{\Sigma}$}
\end{picture}
\caption{Definition of the product}
\label{figure_def_product}
\end{figure}

In our previous papers \cite{TsujiCSAI} \cite{Tsujipurebraid} \cite{TsujiTorelli},
we study a relationship between the Kauffman bracket skein algebra and
the mapping class group on a surface $\Sigma$.
Let $\widehat{\skein{\Sigma}}$ be the completed Kauffman 
bracket skein algebra on $\Sigma$ and
$\widehat{\skein{\Sigma,J}}$  the completed Kauffman 
bracket skein module with base point set $J \times 
\shuugou{\frac{1}{2}}$ for a finite subset $J \subset 
\partial \Sigma$.
In \cite{TsujiCSAI}, we prove the formula of the Dehn twist $t_c$ of a simple closed
curve $c$
\begin{equation*}
t_c ( \cdot) = \exp(\sigma(L(c)))(\cdot) \defeq
\sum_{i=0}^\infty \frac{1}{i!}(\sigma(L(c))^i(\cdot) \in
\Aut (\widehat{\skein{\Sigma,J}})
\end{equation*}
where 
\begin{equation*}
L(c) \defeq \frac{-A+\gyaku{A}}{4 \log(-A)}
(\arccosh (\frac{-c}{2}))^2-(-A+\gyaku{A})
\log(-A).
\end{equation*}
We obtain the formula by analogy of
the formula of the 
completed Goldman Lie algebra
\cite{Kawazumi} \cite{KK} \cite{MT}.
We define the filtration 
$\filtn{F^n \widehat{\skein{\Sigma}}}$
in \cite{Tsujipurebraid}.
We consider $F^3 \widehat{\skein{\Sigma}}$ as 
a group using the Baker Campbell Hausdorff series $\bch$
where $g>1$.
We remark that $\mathcal{I}(\Sigma_{g,1})$ is generated by
$\shuugou{t_{c_1c_2} \defeq t_{c_1}{t_{c_2}}^{-1}|(c_1,c_2) :\mathrm{BP}}$
where a BP (bounding pair) is a pair of two simple closed curves bounding
a submanifold of $\Sigma_{g,1}$.
By analogy of  \cite{KK} 6.3, we have the following.

\begin{thm}[\cite{TsujiTorelli} Theorem 3.13. Corollary 3.14.]
\label{thm_zeta}
The group homomorphism $\zeta :\mathcal{I}(\Sigma_{g,1}) \to
(F^3 \widehat{\skein{\Sigma_{g,1}}}, \bch)$ defined by
$\zeta (t_{c_1c_2}) =L(c_1)-L(c_2)$ for a
BP $(c_1, c_2)$ is injective whre $g>1$.
Furthermore, we have
\begin{equation*}
\xi (\cdot) =\exp (\sigma (\zeta(\xi))(\cdot) \in \Aut(\widehat{\skein{\Sigma_{g,1},J}})
\end{equation*}
for any $\xi \in \mathcal{I}(\Sigma_{g,1})$ and any finite subset $J \subset 
\partial \Sigma_{g,1}$.
\end{thm}
We remark that $\zeta (t_c) =L(c)$ for a separating simple closed curve $c$.

Let $e$ be an embedding $\Sigma_{g,1} \times [0,1]$ satisfying
the following conditions
\begin{align*}
&e_{|\Sigma_{g,1} \times \shuugou{\frac{1}{2}}}:\Sigma \times \shuugou{\frac{1}{2}}
\to \Sigma, (x,\frac{1}{2}) \mapsto x, \\
&e(\Sigma \times [0,\frac{1}{2}]) \subset H_g^+, \\
&e(\Sigma \times [\frac{1}{2},1]) \subset H_g^-. 
\end{align*}
We call this embedding a standard embedding.
We denote by $e_*$ the $\Q[[A+1]]$-module homomorphism
$\widehat{\skein{\Sigma_{g,1}}} \to \Q[[A+1]]$ induced by $e$.
The following is our main theorem.

\begin{thm}
\label{thm_main}
The map $Z: \mathcal{I}(\Sigma_{g,1}) \to \Q[[A+1]]$
defined by 
\begin{equation*}
Z(\xi) \defeq \sum_{i=0}^\infty \frac{1}{(-A+\gyaku{A})^i i!}e_*
((\zeta (\xi))^i)
\end{equation*}
induces
\begin{equation*}
z:\mathcal{H} (3) \to \Q[[A+1]], M(\xi) \to Z(\xi).
\end{equation*}

\end{thm}
\subsection{Main theorem and its proof}

The aim of this subsection is to prove Theorem \ref{thm_main}.

By Proposition \ref{prop_embed_filtration}, the map
$Z: \mathcal{I}(\Sigma_{g,1}) \to \Q[[A+1]]$ is well-defined.

For $\epsilon \in \shuugou{+,-}$,
let $\skein{H_g^\epsilon}$ be the quotient of
$\Q [A,\gyaku{A}] \mathcal{T} (H_g^\epsilon)$
by the skein relation and the trivial knot relation,
where $\mathcal{T}(H_g^\epsilon)$
is the set of unoriented framed link
in $H_g^\epsilon$. We can consider  its completion
$
\widehat{\skein{H_g^\epsilon}}$,
for details see \cite{TsujiCSAI} Theorem 5.1.
We denote the embedding $\iota'^+ :
\Sigma_{g,1}  \times [0,\frac{1}{2}] \to{H_g^+}$
and the embedding $\iota'^- :
\Sigma_{g,1}  \times [\frac{1}{2},1] \to {H_g^-}$.
The embeddings
\begin{align*}
&\iota^+:\Sigma_{g,1} \times I  \to H_g^+, (x,t) \mapsto \iota'^+(x,t/2), \\
&\iota^-:\Sigma_{g,1} \times I  \to H_g^-, (x,t) \mapsto \iota'^+(x,(t+1)/2) \\
\end{align*}
induces
\begin{align*}
&\iota^+:\widehat{\skein{\Sigma_{g,1}}} \to \widehat{\skein{H_g^+}}, \ \
\iota^-:\widehat{\skein{\Sigma_{g,1}}} \to \widehat{\skein{H_g^-}}. \\
\end{align*}

By definition, we have the followings.

\begin{prop}
\label{prop_ideal}
\begin{enumerate}
\item The kernel of $\iota^+$ is a right ideal
of $\widehat{\skein{\Sigma_{g,1}}}$.
\item The kernel of $\iota^-$ is a left ideal
of $\widehat{\skein{\Sigma_{g,1}}}$.
\item We have $e_* (\ker \iota^\epsilon) =\shuugou{0}$
for $\epsilon \in \shuugou{+,-}$.
\end{enumerate}
\end{prop}

\begin{prop}
\label{prop_bch_Z}
We have $Z(\xi_1 \xi_2) =\sum_{i,j \geq 0}\dfrac{1}{(-A+A^{-1})^{i+j} i! j!}e_* ((\zeta(\xi_1))^i
(\zeta(\xi_2))^j)$.
\end{prop}

\begin{lemm}
\label{lemm_proof_key}
\begin{enumerate}
\item 
Let $c_a$, $c'_a$, $c_b$ and $c'_b$ be 
simple closed curves as in Figure \ref{figure_handle_torelli_generator}.
For $\epsilon \in \shuugou{+,-}$, 
we have 
\begin{align*}
\iota^\epsilon (\xi(L(c_a)-L(c'_a))=0,
\iota^\epsilon (\xi(L(c_b)-L(c'_b))=0
\end{align*}
for $\xi \in \mathcal{M}(H_{g,1}^\epsilon)$.
\item 
Let $c_{a,i}$, $c_{b,j}$ and $c_{i,j}$
be simple closed curves as in
Figure \ref{figure_handle_G_S_j_i} or Figure \ref{figure_handle_G_S_i_j}.
For $\epsilon \in \shuugou{+,-}$, 
we have 
\begin{align*}
\iota^\epsilon (L(c_{i,j})-L(c_{a,i})-L(c_{b,j}))=0
\end{align*}for $i \neq j$.
\end{enumerate}

\end{lemm}

By Lemma \ref{lemm_torelli_check},
in order to prove Theorem \ref{thm_main},
it is enough to check the following lemmas.

\begin{lemm}
For any $i$, we have $e_* \circ h_i =e_*$.
Furthermore, we have $Z (h_i \xi h_i^{-1}) =Z (\xi)$.
\end{lemm}
\begin{proof}
The embeddings $e \circ h_i $ and $e$ are
isotopic. This proves the first claim.
Using this, we have $e_* (\zeta (h_i \xi h_i^{-1}))=e_* \circ h_i (\zeta(\xi))$.
This proves the second claim.
This proves the lemma.
\end{proof}

\begin{lemm}
For any $i \neq j$, we have $e_* \circ s_{ij} =e_*$.
Furthermore, we have $Z (s_{ij }\xi s_{ij}^{-1}) =Z (\xi)$.
\end{lemm}

\begin{proof}
We fix an element $x$ of $\widehat{\skein{\Sigma_{g,1}}}$.
We have $s_{ij} (x) =\exp (\sigma(L(c_{i,j})-L(c_{a,i})-L(c_{b,j})))(x)$.
Using Lemma \ref{lemm_proof_key}(2) and Proposition
\ref{prop_ideal} (1)(2)(3),
we have $e_*(\exp (\sigma(L(c_{i,j})-L(c_{a,i})-L(c_{b,j})))(x))=e_*(x)$.
This proves the first claim.
Using this, we have $e_* (\zeta (s_{ij} \xi s_{ij}^{-1}))=e_* \circ s_{ij} (\zeta(\xi))$.
This proves the second claim.
This  proves the lemma.
\end{proof}

\begin{lemm}
\begin{enumerate}
\item We have $Z(\xi \eta^+) =Z(\xi)$ for any 
$\xi \in \mathcal{I}(\Sigma_{g,1})$ and 
any $\eta^+ \in \shuugou{t_{\xi(c_a) \xi( c'_a)}|\xi \in \mathcal{M} (H_{g,1}^+)}
\cup \shuugou{t_{\xi(c_b) \xi(c'_b)}|\xi \in \mathcal{M} (H_{g,1}^+)}$.
\item We have $Z(\eta^- \xi ) =Z(\xi)$ for any 
$\xi \in \mathcal{I}(\Sigma_{g,1})$ and 
any $\eta^- \in \shuugou{t_{\xi(c_a) \xi( c'_a)}|\xi \in \mathcal{M} (H_{g,1}^-)}
\cup \shuugou{t_{\xi(c_b) \xi(c'_b)}|\xi \in \mathcal{M} (H_{g,1}^-)}$.
\end{enumerate}
\end{lemm}

\begin{proof}
We prove only (1) (because the proof of (2) is almost the same.)
By Proposition \ref{prop_bch_Z}, we have 
$Z(\xi \eta^+) =\sum_{i,j \geq 0}\dfrac{1}{(-A+A^{-1})^{i+j} i! j!}e_* ((\zeta(\xi))^i
(\zeta(\eta^+))^j)$.
Using Lemma \ref{lemm_proof_key} and Proposition \ref{prop_ideal} (1)(3),
we obtain
\begin{equation*}
Z(\xi \eta^+) =\sum_{i,j \geq 0}\frac{1}{(-A+A^{-1})^{i+j} i! j!}e_* ((\zeta(\xi))^i
(\zeta(\eta^+))^j)=\sum_{i \geq 0}\frac{1}{(-A+A^{-1})^{i} i!}e_* ((\zeta(\xi))^i)=Z(\xi).
\end{equation*}
This proves the lemma.
\end{proof}

\begin{proof}[Proof of Theorem \ref{thm_main}]
By Fact \ref{fact_map_3_manifold} and
Lemma \ref{lemm_torelli_check}, it is enough to check the following
\begin{align*}
Z(h_i \xi {h_i}^{-1}) =Z(\xi), \\
Z(s_{ij} \xi {s_{ij}}^{-1}) =Z(\xi), \\
Z(\xi \eta^+) =Z(\xi), \\
Z(\eta^- \xi)=Z(\xi),
\end{align*}
for any $\xi \in \mathcal{I}(\Sigma_{g,1})$,
any $\eta^+ \in \shuugou{t_{\xi(c_a) \xi( c'_a)}|\xi \in \mathcal{M} (H_{g,1}^+)}
\cup \shuugou{t_{\xi(c_b) \xi(c'_b)}|\xi \in \mathcal{M} (H_{g,1}^+)}$,
any $\eta^- \in \shuugou{t_{\xi(c_a) \xi( c'_a)}|\xi \in \mathcal{M} (H_{g,1}^-)}
\cup \shuugou{t_{\xi(c_b) \xi(c'_b)}|\xi \in \mathcal{M} (H_{g,1}^-)}$
and any $i \neq j$.
The above lemmas prove
these equations.
This proves the Theorem.

\end{proof}

\bigskip

This invariant satisfies the following conditions.

\begin{prop}
\label{prop_z_disjoint}
For $M_1, M_2 \in \mathcal{H}(3)$, we have
\begin{equation*}
z(M_1 \sharp M_2) =z(M_1)z(M_2)
\end{equation*}
where $M_1 \sharp M_2 $ is the connected sum of $M_1$ and $M_2$.
\end{prop}

\begin{proof}
Let $\iota_1:\Sigma^1 \to \Sigma_{g,1}$ and $\iota_2:\Sigma^2 \to \Sigma_{g,1}$
 be the embedding maps as in Figure \ref{figure_Sigma_disjoint}.
The embedding maps induces 
\begin{align*}
&\iota_1:\mathcal{M}(\Sigma^1) \to \mathcal{M}(\Sigma_{g.1}), 
&\iota_2:\mathcal{M}(\Sigma^2) \to \mathcal{M}(\Sigma_{g,1}), \\
&\iota_1:\skein{\Sigma^1} \to \skein{\Sigma_{g,1}}, 
&\iota_2:\skein{\Sigma^2} \to \skein{\Sigma_{g,2}}.
\end{align*}
We remark
$e_* (\iota_1(x_1) \iota_2(x_2))=e_*(\iota_1(x_1))e_*(\iota_2(x_2))$
for $x_1 \in \skein{\Sigma^1}$ and $x_2 \in \skein{\Sigma^2}$.
For $\xi_1 \in \iota_1(\mathcal{M}(\Sigma^1))$
and $\xi_2 \in \iota_2(\mathcal{M}(\Sigma^2))$
 we have
\begin{align*}
&z(M(\xi_1)\sharp M(\xi_2))=z(M(\xi_1 \circ \xi_2)) 
=Z(\xi_1 \circ \xi_2)=\sum_{i=1}^\infty \frac{1}{(-A+A^{-1})^i i!}e_*((\zeta(\xi_1 \circ \xi_2))^i) \\
&=\sum_{i,j \geq 0} \frac{1}{(-A+A^{-1})^{i+j} i!j!}e_*((\zeta(\xi_1))^i (\zeta(\xi_2))^j) 
=\sum_{i,j \geq 0} \frac{1}{(-A+A^{-1})^{i+j} i!j!}e_*((\zeta(\xi_1))^i )e_*((\zeta(\xi_2))^j) \\
&=(\sum_{i=0}^\infty\frac{1}{(-A+A^{-1})^i i!}e_*((\zeta(\xi_1))^i))(\sum_{j=0}^\infty\frac{1}{(-A+A^{-1})^j j!}e_*((\zeta(\xi_2))^j))=z(M(\xi_1))z(M(\xi_2)).
\end{align*}
This proves the proposition.

\begin{figure}
\begin{picture}(300,140)
\put(0,0){\includegraphics[width=300pt]{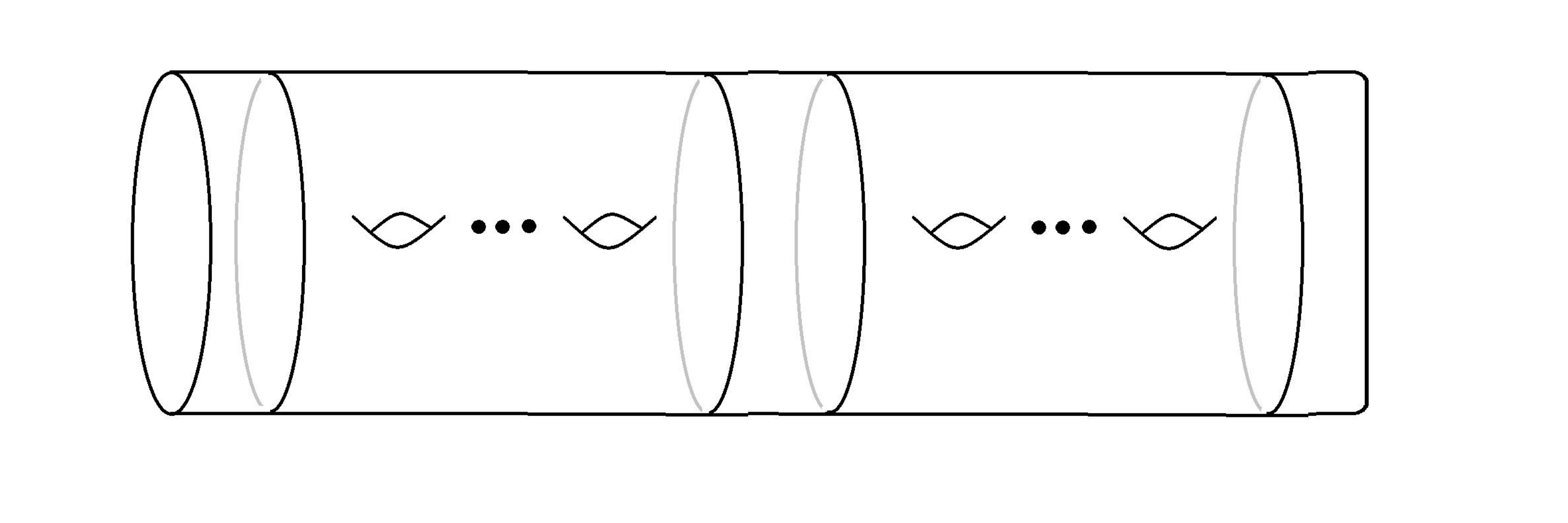}}
\put(100,70){$\Sigma^1$}
\put(200,70){$\Sigma^2$}
\end{picture}
\caption{$\Sigma^1$ and $\Sigma^2$}
\label{figure_Sigma_disjoint}
\end{figure}
\end{proof}

\begin{prop}
For $\xi_1 \in \zeta^{-1} (F^{n_1+2}\widehat{\skein{\Sigma_{g,1}}}), 
\cdots, \xi_k \in \zeta^{-1}(F^{n_k+2} \widehat{\skein{\Sigma_{g,1}}})$,
we have
\begin{equation*}
\sum_{\epsilon_i \in \shuugou{1,0}}(-1)^{\sum \epsilon_i}
z (M(\xi_1^{\epsilon_1} \cdots \xi_k^{\epsilon_k}))
\in (A+1)^{\gauss{(n_1 + \cdots +n_k+1)/{2}}}\Q [[A+1]].
\end{equation*}
We remark that $\zeta^{-1} (F^3 \widehat{\skein{\Sigma_{g,1}}})$
equals $\mathcal{I} (\Sigma_{g,1})$
and
that $\zeta^{-1}(F^4 \widehat{\skein{\Sigma_{g,1}}})$
equals the Johnson kernel.
\end{prop}

\begin{proof}
We have
\begin{align*}
&\sum_{\epsilon_i \in \shuugou{1,0}}(-1)^{\sum \epsilon_i}
z (M(\xi_1^{\epsilon_1} \cdots \xi_k^{\epsilon_k})) \\
&=e_*((1-\exp (\frac{\zeta(\xi_1)}{-A+\gyaku{A}}))\cdots(1-\exp (\frac{\zeta(\xi_k)}{-A+\gyaku{A}})).
\end{align*}
By Proposition \ref{prop_embed_filtration}, we have
\begin{equation*}
\sum_{\epsilon_i \in \shuugou{1,0}}(-1)^{\sum \epsilon_i}
z (M(\xi_1^{\epsilon_1} \cdots \xi_k^{\epsilon_k}))
\in (A+1)^{\gauss{({n_1 + \cdots +n_k+1})/{2}}}\Q [[A+1]].
\end{equation*}
This proves the proposition.

\end{proof}

\begin{cor}
\label{cor_finite_type}
The invariant $z(M) \in \Q[[A+1]]/((A+1)^{n+1})$
is a finite type invariant for $M \in \mathcal{H} (3)$
order $n$.
\end{cor}

\begin{prop}
\label{prop_z_finite_nontrivial}
The invariant $z(M) \in \Q[[A+1]]/((A+1)^{n+1})$
is a nontrivial finite type invariant for $M \in \mathcal{H} (3)$
order $n$.
\end{prop}

\begin{proof}
It is enough to show that
 there exists $\xi_1, \cdots , \xi_n \in \mathcal{K}(\Sigma_{g,1})$
such that 
\begin{equation*}
\sum_{\epsilon_i \in \shuugou{0,1}} (-1)^{\sum \epsilon_i}
M(\prod_{i=1}^n {\xi_i}^{\epsilon_i}) \neq 0 \mod ((A+1)^{n+1}).
\end{equation*}
Let $c^1_1,c^1_2,c^1_L, \cdots
c^n_1, c^n_2, c^n_L$
be simple closed curves as in Figure \ref{figure_nontrivial_Sigma}.
We denote $c_i \defeq t_{c^i_1} \circ t_{c^i_2} (c^i_L)$ and
$t_i \defeq t_{c_i}$
for $i =1,  \cdots, n$.
We remark that $M(t_{i_1}\circ \cdots \circ t_{i_k})
=\sharp^k M(t_1)$ for $1 \leq i_1 <\cdots <i_k \leq n$
and that $M(t_1)$ is the \Poincare homology 3-sphere. 
By the computation in section
\ref{section_example}
and Proposition \ref{prop_z_disjoint},
\begin{equation*}
\sum_{\epsilon_i \in \shuugou{0,1}} (-1)^{\sum \epsilon_i}
M(\prod_{i=1}^n {t_i}^{\epsilon_i})
=(1-z(M(t_1)))^n=6^n(A^4-1)^n \mod ((A+1)^{n+1}).
\end{equation*}

This proves the proposition.

\begin{figure}
\begin{picture}(400,140)
\put(0,0){\includegraphics[width=400pt]{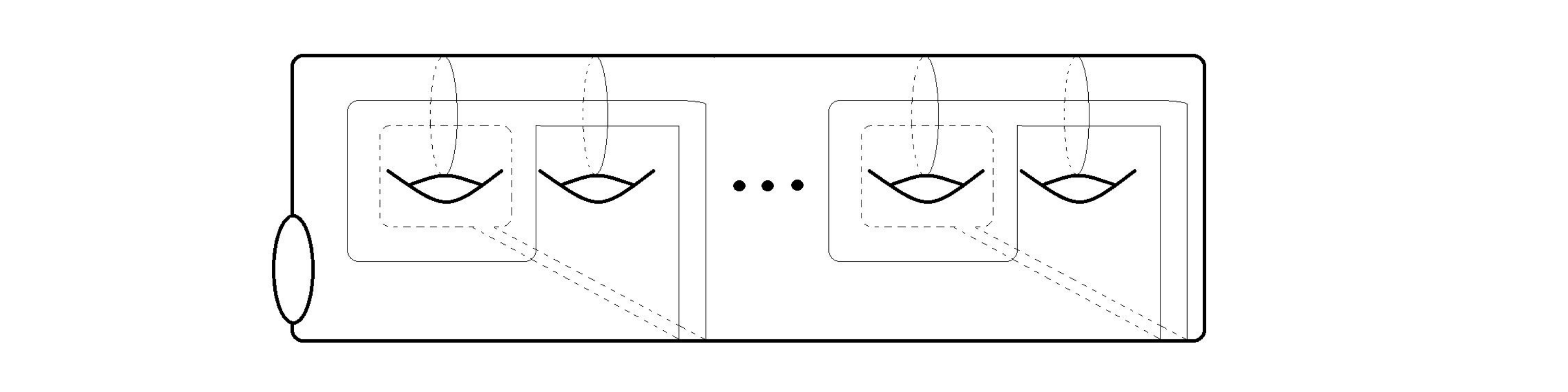}}
\put(100,90){$c^1_1$}
\put(150,90){$c^1_2$}
\put(100,20){$c^1_L$}
\put(230,90){$c^n_1$}
\put(280,90){$c^n_2$}
\put(230,20){$c^n_L$}
\end{picture}
\caption{$c^1_1,c^1_2,c^1_L, \cdots
c^n_1, c^n_2, c^n_L$}
\label{figure_nontrivial_Sigma}
\end{figure}
\end{proof}

By the computation in section
\ref{section_example}, the coefficient of $A^4-1$ in 
the invariant $z$ for the \Poincare homology $3$-sphere
is $-6$.
Since the casson invariant is the unique 
nontrivial finite type invariant of order $1$
upto a scalar, we have the following.

\begin{cor}
For any $M \in \mathcal{H}(3)$,
we have the coefficeint of $A^4-1$ in $z(M)$
is $(-6)$ times the Casson invariant.
\end{cor}

\section{Example}
\label{section_example}

Let $c_1$, $c_2$ and $c_L$ be simple closed curves in $\Sigma_{g,1}$ as in 
Figure \ref{figure_handle_example}.
We consider integral homology 3-spheres
$M(\epsilon_1,\epsilon_2,\epsilon_3) \defeq 
M ((t_{{t_1}^{\epsilon_1} \circ {t_2}^{\epsilon_2} (c_L)})^{\epsilon_3})$
for $\epsilon_1, \epsilon_2, \epsilon_3 \in \shuugou{\pm 1}$,
where $t_1 \defeq t_{c_1}$ and $t_2  \defeq t_{c_2}$.
For $\epsilon \in  \shuugou{\pm 1}$,
the manifold $M(1,-1,\epsilon) \simeq M(-1,1,\epsilon)$  is 
the integral homology 3-sphere
obtained from $S^3$ performing  the $\epsilon$-surgery on the figure eight knot $4_1$,
which is $e (t_1^{-1} \circ t_2 (c_L))$.
For $\epsilon  \in \shuugou{\pm 1}$,
the manifold $M(1,1,\epsilon)$  is 
the integral homology 3-sphere
obtained from $S^3$ performing  the $\epsilon$-surgery on the trefoil knot $3_1$,
which is $e (t_1 \circ t_2 (c_L))$.
For $\epsilon \in \shuugou{\pm 1}$,
the manifold $M(-1,-1,\epsilon)$  is 
the integral homology 3-sphere
obtained from $S^3$ performing  the $\epsilon$-surgery on the mirror  $-3_1$
of the trefoil knot,
which is $e (t_1^{-1} \circ t_2^{-1} (c_L))$.
In particular, $M(1,1,1)$ is the \Poincare homology sphere.
We remark that $M(1,-1,1)$ and $M(-1,-1,-1)$
are the same $3$-manifold.
By straight forward computations using Habiro's formula \cite{Habiro2000}
for colored Jones polynomials of the trefoil knot and the figure eight knot,
 we have the following.
We remark that we compute
$z(M(1,-1,1))=z(M(1,1,-1)) $ (rep. 
$z(M(1,-1,-1))=z(M(-1,-1,1))$)  by two ways
$Z(t_{{t_1} \circ {t_2}^{-1} (c_L)})=
Z((t_{{t_1} \circ {t_2} (c_L)})^{-1})$
(resp. $Z((t_{{t_1} \circ {t_2}^{-1} (c_L)})^{-1})
=Z(t_{{t_1}^{-1} \circ {t_2}^{-1} (c_L)})$.

\begin{prop}
We have 
\begin{align*}
z(M(1,1,1)) 
&=[ 1, -6,45,-464,6224,-102816,2015237, \\
&-45679349,1175123730,-33819053477,  \\
&1076447743008, -37544249290614, \\
&1423851232935885,-58335380481272491], \\
z(M(1,-1,1))=z(M(1,1,-1) 
&=[ 1,6,63,932,17779,415086,11461591,365340318, \\
&13201925372,533298919166,23814078531737, \\
&1164804017792623,61932740213389942, \\
&3556638330023177088], \\
z(M(-1,-1,-1)) 
&=[1, 6, 39, 380,  4961,  80530,  1558976,  35012383, \\
 &894298109,  25591093351,  810785122236, \\
&28169720107881,  1064856557864671,  \\
&43506118030443092], \\
z(M(1,-1,-1))=z(M(-1,-1,1) 
&=[  1, -6,  69,  -1064,  20770,  -492052,  13724452,  \\
&-440706098,  16015171303,  -649815778392,  29121224693198,  \\
&-1428607184648931,  76147883907835312,   \\
&-4382222160786508572].
\end{align*}
Here we denote 
\begin{equation*}
1+a_1 (A^4-1) +a_2(A^4-1)^2 + \cdots+a_{13}(A^4-1)^{13}+o(14)
=[1,a_1, a_2, \cdots,a_{13}]
\end{equation*}
where $o(14) \in (A+1)^{14} \Q [[A+1]]$.
\end{prop}

\begin{figure}
\begin{picture}(300,140)
\put(0,0){\includegraphics[width=300pt]{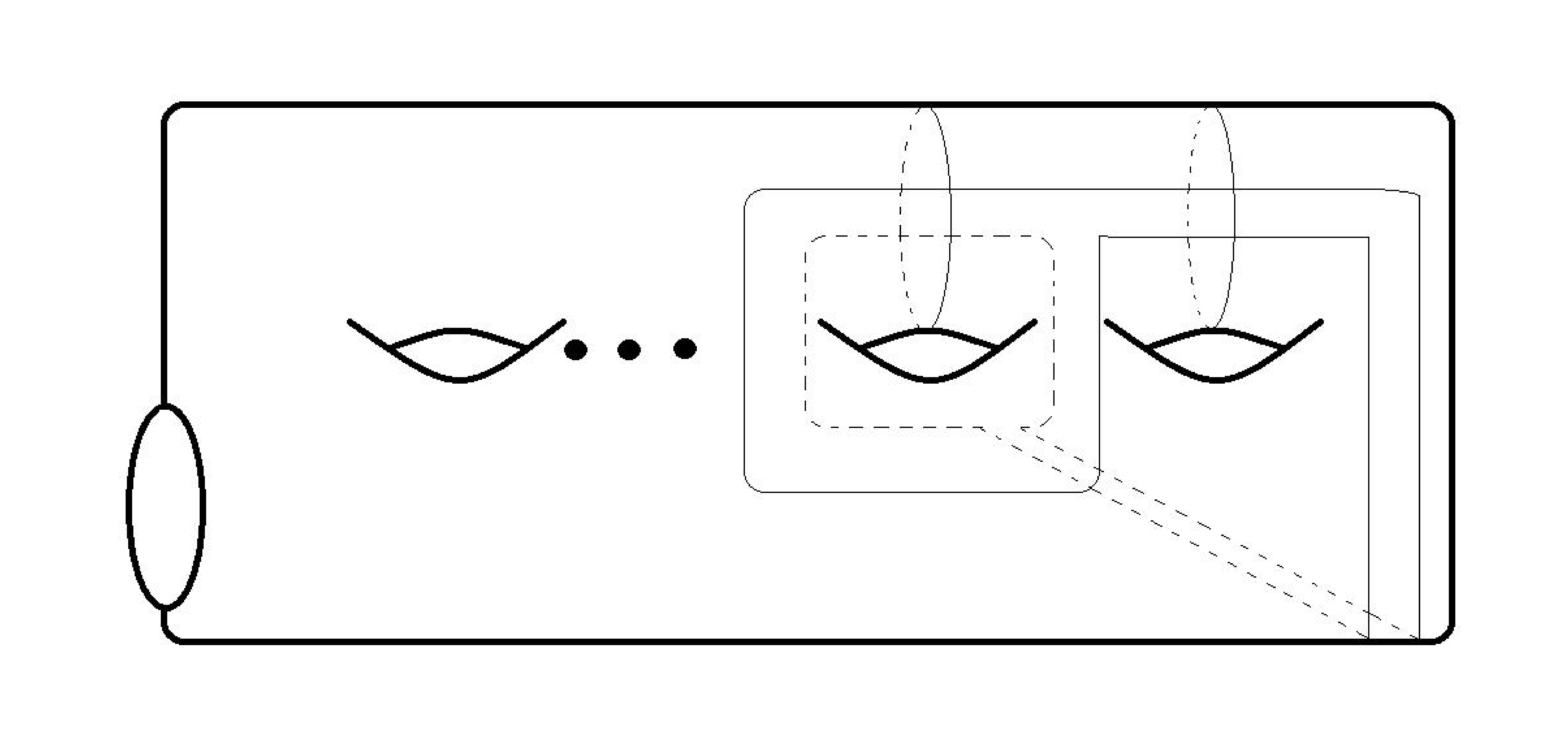}}
\put(37,90){$H_g^+$}
\put(230,130){$c_1$}
\put(180,40){$c_L$}
\put(177,130){$c_2$}
\put(14,90){$H_g^-$}
\end{picture}
\caption{simple closed curves $c_1$, $c_2$ and $c_L$}
\label{figure_handle_example}
\end{figure}

\end{document}